\documentclass[12pt,english]{article}
\usepackage[T1]{fontenc}
\usepackage[latin9]{inputenc}
\usepackage{geometry}
\usepackage{mathrsfs}
\usepackage{amscd}
\usepackage{amsmath}
\usepackage{amsthm}
\usepackage{amssymb}

\makeatletter

\providecommand{\tabularnewline}{\\}

\numberwithin{equation}{section}
\numberwithin{figure}{section}
\theoremstyle{plain}
\newtheorem{thm}{\protect\theoremname}[section]
\theoremstyle{plain}
\newtheorem{assumption}[thm]{\protect\assumptionname}
\theoremstyle{plain}
\newtheorem{lem}[thm]{\protect\lemmaname}
\theoremstyle{definition}
\newtheorem*{example*}{\protect\examplename}
\theoremstyle{definition}
\newtheorem{defn}[thm]{\protect\definitionname}
\theoremstyle{plain}
\newtheorem{prop}[thm]{\protect\propositionname}
\theoremstyle{plain}
\newtheorem{cor}[thm]{\protect\corollaryname}
\theoremstyle{plain}
\newtheorem*{prop*}{\protect\propositionname}
\theoremstyle{remark}
\newtheorem{rem}[thm]{\protect\remarkname}

\date{}

\makeatother

\usepackage{babel}
\providecommand{\assumptionname}{Assumption}
\providecommand{\corollaryname}{Corollary}
\providecommand{\definitionname}{Definition}
\providecommand{\examplename}{Example}
\providecommand{\lemmaname}{Lemma}
\providecommand{\propositionname}{Proposition}
\providecommand{\remarkname}{Remark}
\providecommand{\theoremname}{Theorem}

\begin{document}
\title{On Radon hypergeometric functions on the Grassmannian manifold }
\author{Hironobu Kimura,\\
 Department of Mathematics, Graduate School of Science and\\
 Technology, Kumamoto University}

\maketitle

\global\long\def\R{\mathbb{R}}%
 
\global\long\def\al{\alpha}%
\global\long\def\be{\beta}%
 
\global\long\def\ga{\gamma}%
 
\global\long\def\de{\delta}%
 
\global\long\def\expo{\mathrm{exp}}%
 
\global\long\def\f{\varphi}%
 
\global\long\def\W{\Omega}%
 
\global\long\def\wm{\omega}%
 
\global\long\def\lm{\lambda}%
\global\long\def\te{\theta}%
 
\global\long\def\C{\mathbb{C}}%
 
\global\long\def\Z{\mathbb{Z}}%
 
\global\long\def\Ps{\mathbb{P}}%
 
\global\long\def\De{\Delta}%
 
\global\long\def\cbatu{\mathbb{C}^{\times}}%
 
\global\long\def\La{\Lambda}%
 
\global\long\def\vt{\vartheta}%
 
\global\long\def\G{\Gamma}%
\global\long\def\GL#1{\mathrm{GL}(#1)}%
 
\global\long\def\Span{\mathrm{span}}%

\global\long\def\gras{\mathrm{Gr}}%
 
\global\long\def\fl{\mathrm{Flag}}%
 
\global\long\def\ep{\varepsilon}%
  
\global\long\def\diag{\mathrm{diag}}%
 
\global\long\def\tr{\,\mathrm{^{t}}}%
 
\global\long\def\re{\mathrm{Re}}%
 
\global\long\def\im{\mathrm{Im}}%
 
\global\long\def\norm{\mathscr{N}(n)}%
 
\global\long\def\sm{\sigma}%
 
\global\long\def\ini{\mathrm{in}_{\prec}}%
 
\global\long\def\cL{\mathcal{L}}%
 
\global\long\def\rank{\mathrm{rank}}%
 
\global\long\def\tH{\tilde{H}}%
 
\global\long\def\mat{\mathrm{Mat}}%
 
\global\long\def\lto{\longrightarrow}%
 
\global\long\def\Si{\mathfrak{S}}%
 
\global\long\def\cO{\mathcal{O}}%
 
\global\long\def\gl{\mathfrak{gl}}%
\global\long\def\gee{\mathfrak{g}}%
 
\global\long\def\Tr{\,\mathrm{Tr}}%
 
\global\long\def\ad{\mathrm{ad}}%
 
\global\long\def\Ad{\mathrm{Ad}}%
 
\global\long\def\ha{\mathfrak{h}}%
 
\global\long\def\fj{\mathfrak{j}}%
 
\global\long\def\mrn{\mat'(r,N)}%
 
\global\long\def\sgn{\mathrm{sgn}}%
 
\global\long\def\hlam{H_{\lambda}}%
 
\global\long\def\cR{\mathcal{R}}%
 
\global\long\def\mnm{\mat'(m,N)}%
\global\long\def\cP{\mathcal{P}}%
 
\global\long\def\herm{\mathscr{H}(r)}%
 
\global\long\def\ghyp{\,_{2}F_{1}}%
 
\global\long\def\kum{\,_{1}F_{1}}%
 
\global\long\def\auto{\mathrm{Aut}}%
 
\global\long\def\la{\langle}%
 
\global\long\def\ra{\rangle}%
 
\global\long\def\Ai{\mathrm{Ai}}%
 
\global\long\def\adj{\mathrm{Ad}}%
 
\global\long\def\yn{\mathbf{Y}_{n}}%
 
\global\long\def\eq{\mathcal{I}}%
 
\global\long\def\hyp#1#2{\, _{#1}F_{#2}}%
 
\global\long\def\jro{J_{r}^{\circ}}%
 
\global\long\def\jroo{\mathfrak{j}_{r}^{\circ}}%
\global\long\def\jor#1{J^{\circ}(#1)}%
 
\global\long\def\fa{\mathfrak{a}}%
 
\global\long\def\pa{\partial}%
 
\global\long\def\bx{\mathbf{x}}%
 
\global\long\def\etr{\mathrm{etr}}%

\begin{abstract}
We give a definition of Radon hypergeometric function (Radon HGF)
of confluent and nonconfluent type, which is a function on the Grassmannian
$\gras(m,nr)$ obtained as a Radon transform of a character of the
universal covering group of $H_{\lm}\subset\GL{nr}$ specified by
a partition $\lm$ of $n$, where $H_{(1,\dots,1)}\simeq(\GL r)^{n}$.
When $r=1$, the Radon HGF reduces to the Gelfand HGF on the Grassmannian.
We give a system of differential equations satisfied by the Radon
HGF and show that the Hermitian matrix integral analogues of Gauss
HGF and its confluent family: Kummer, Bessel, Hermite-Weber and Airy
function, are obtained in a unified manner as the Radon HGF on $\gras(2r,4r)$
corresponding to the partitions $(1,1,1,1)$, $(2,1,1)$, $(2,2)$,
$(3,1)$ and $(4)$, respectively.
\end{abstract}

\section{Introduction}

This paper concerns a complex analytic study of the hypergeometric
function (HGF) on the Grassmannian manifold $\gras(m,N)$, the set
of $m$-dimensional subspaces of $\C^{N}$, defined by a Radon transform,
which we call the Radon hypergeometric function in this paper. 

Let us explain our motivation and reason to consider such functions.
In the theory of special functions, the Gauss HGF is important and
maybe the best known one, which can be defined by the integral 
\begin{equation}
\ghyp(a,b,c;x)=\frac{\G(c)}{\G(a)\G(c-a)}\int_{0}^{1}u^{a-1}(1-u)^{c-a-1}(1-ux)^{-b}du.\label{eq:intro-1}
\end{equation}
Note that the integrand is a product of complex powers of polynomials
in $u$ of degree $1$. In \cite{Aomoto}, Aomoto started the study
of functions defined by an integral 
\begin{equation}
F(x)=\int_{C}\prod_{j=1}^{N}f_{j}(u,x)^{\al_{j}}du_{1}\wedge\cdots\wedge du_{m-1},\quad f_{j}=x_{0,j}+u_{1}x_{1,j}+\cdots+u_{m-1}x_{m-1,j},\label{eq:intro-2}
\end{equation}
where $f_{1}$ is sometimes normalize as $f_{1}=1$. To obtain the
Gauss HGF, we consider (\ref{eq:intro-2}) for $(m,N)=(2,4)$ and
$f_{1}=1,f_{2}=u,f_{3}=1-u,f_{4}=1-ux$. Interpreting the work of
Aomoto in terms of Radon transform, Gelfand \cite{Gelfand} introduced
a class of HGF on the Grassmannian manifold $\gras(m,N)$, called
nowadays Gelfand's HGF on the Grassmannian. For a Cartan subgoup $H\subset\GL N$
consisting of diagonal matrices, let $\chi:\tilde{H}\to\cbatu$ be
a character of the universal covering group $\tilde{H}$ of $H$ which
is given by $\chi(h;\al)=\prod_{j=1}^{N}h_{j}^{\al_{j}}$ for $h=\diag(h_{1},\dots,h_{N})$.
Then the Radon transform of $\chi(h;\al)$ is considered; it is, roughly
speaking, to substitute linear forms $h_{j}(t)=\sum_{0\leq i<m}t_{i}x_{i,j}$
into $\chi$, where $h_{j}(t)$ is regarded as a form on the space
of homogeneous coordinates of $\Ps^{m-1}$, and then to integrate
it on some $(m-1)$-chain $C$ of $\Ps^{m-1}$. If we use the affine
coordinates $u=(u_{1},\dots,u_{m-1}),\;u_{i}=t_{i}/t_{0}$ in the
affine open set $\{[t]\in\Ps^{m-1}\mid t_{0}\neq0\}$, we obtain (\ref{eq:intro-2}).
In this context, the Gauss HGF can be understood as Gelfand's HGF
on $\gras(2,4)$. A confluent family of Gauss HGF is also important
in various domains of mathematics. The members of confluent family
which we consider are Kummer's confluent HGF, Bessel function, Hermite-Weber
function and Airy function \cite{IKSY}. For example, the Airy function
is given by the integral 
\[
\Ai(x)=\int_{C}\exp\left(ux-\frac{1}{3}u^{3}\right)du,
\]
and is considered as the most simple complex oscillatory integral
whose phase function is a versal deformation of $A_{2}$ type simple
singularity \cite{Arnold}. We have a framework of understanding these
functions as functions on $\gras(2,4)$ in the spirit of Gelfand \cite{Kimura-Haraoka,Kimura-H-T,Kimura-Koitabashi}.
In fact, to define the Gelfand HGF of confluent type we use a maximal
abelian subgroup $H_{\lm}\subset\GL N$ obtained as a centralizer
of a regular element of $\GL N$ specified by a partition $\lm$ of
$N$. As will be seen in Section \ref{subsec:herm-radon-2}, the confluent
family of Gauss: Kummer, Bessel, Hermite-Weber and Airy are essentially
the same as the Gelfand HGF on $\gras(2,4)$ corresponding to the
partitions $(2,1,1)$, $(2,2)$, $(3,1)$ and $(4)$, respectively.

On the other hand, there is another extension of Gauss HGF defined
by a Hermitian matrix integral
\[
\frac{\G_{r}(c)}{\G_{r}(a)\G_{r}(c-a)}\int_{0<U<1_{r}}(\det U)^{a-r}(\det(1_{r}-U))^{c-a-r}(\det(1_{r}-UX))^{-b}\,dU,
\]
where $\G_{r}(a)$ is a Hermitian matrix integral analogue of the
gamma function and the integration is done in the space $\herm$ of
$r\times r$ Hermitian matrices with the standard Euclidean volume
form $dU$. The integral gives a function of $X\in\herm$. A Hermitian
matrix integral analogue of the confluent family of Gauss also exists
and studied in various context \cite{Faraut,inamasu-ki,kimura-1,Kontsevich,Mehta,muirhead},
in multivariate statistics and Random matrix theory, for example.
Especially, it should be noticed that the Hermitian matrix integral
analogue of Airy
\[
\int_{C}\exp\left(\Tr\left(UX-\frac{1}{3}U^{3}\right)\right)dU,
\]
is used by Kontsevich \cite{Kontsevich} to solve Witten's conjecture
on the $2$-dimensional quantum gravity. For the explicit form of
the Hermitian matrix integral analogues, see Section \ref{subsec:class-HGF-matrix-2}
and \cite{inamasu-ki,kimura-1}.

Taking the above situation into account, it is natural to ask if these
Hermitian matrix integrals can also be understood in a similar way
as the classical HGF is understood as Gelfand's HGF. The first attempt
was done by Oshima \cite{Ohsima}, who considered an extension of
Gelfand HGF mainly for the non-confluent type from the viewpoint of
representation theory. Tanisaki \cite{Tanisaki} gave an extension
of Gelfand's idea to define a HGF on Hermitian symmetric spaces by
Radon transform on the level of $D$-modules. The HGF, discussed in
the works of Oshima and Tanisaki, will be called the Radon HGF.

In this paper, we discuss the Radon HGF on the Grassmannian from the
point of view of complex analysis and make apparent its relation to
the Hermitian matrix integrals mentioned above.

The paper is organized as follows. In Section \ref{sec:Definition-of-Radon},
we give the definition of Radon HGF of non-confluent type and confluent
type on the Grassmannian $\gras(m,nr)$ as a Radon transform of a
character of the subgroup $H_{\lm}\subset\GL{nr}$ specified by a
partition $\lm$ of $n$. For example, the group $H_{(1,\dots,1)}$
is isomorphic to $(\GL r)^{n}$. In Section \ref{sec:Radon-hypergeometric-system},
the system of differential equations for the Radon HGF is given. The
main body of the system is differential equations of order $r+1$
which characterizes the image of Radon transform. It is an easy part
to show that an image of Radon transform satisfies such $(r+1)$-th
order differential equations. For the completeness of presentation,
we give its proof. In Section \ref{sec:Radon-HGF-as}, we give an
explicit form of the Hermitian matrix integral analogue of Gauss,
Kummer, Bessel, Hermite-Weber and Airy, and show that they are obtained
as a Radon HGF on $\gras(2r,4r)$ corresponding to the partitions
$(1,1,1,1)$, $(2,1,1)$, $(2,2)$, $(3,1)$ and $(4)$, respectively. 

\section{\label{sec:Definition-of-Radon}Definition of Radon HGF}

\subsection{Rough sketch of Radon transform}

Let $r,m$ and $N$ be positive integers such that $r<m<N$, $V=\C^{N}$
and $\fl(r,m,V)$ be a generalized flag manifold:
\[
\fl(r,m,V):=\left\{ (v_{1},v_{2})\mid v_{1}\subset v_{2}\subset V:\text{subspaces, \ensuremath{\dim_{\C}v_{1}=r,\dim_{\C}v_{2}=m}}\right\} .
\]
It is a smooth complex algebraic variety of dimension $r(m-r)+m(N-m)$.
Consider a double fibration 
\[
\begin{array}{ccccc}
 &  & \fl(r,m,V)\\
 & \pi_{1}\swarrow &  & \searrow\pi_{2}\\
\\
 & M_{1}:=\gras(r,V) &  & M_{2}=\gras(m,V)
\end{array}
\]
where $\gras(r,V)$ is the Grassmannian manifold of $r$-dimensional
complex subspaces in $V$ and $\pi_{1},\pi_{2}$ are projections defined
by 
\begin{align*}
\pi_{1} & :\fl(r,m,V)\ni(v_{1},v_{2})\mapsto v_{1}\in\gras(r,V),\\
\pi_{2} & :\fl(r,m,V)\ni(v_{1},v_{2})\mapsto v_{2}\in\gras(m,V).
\end{align*}
Sometimes $\gras(r,\C^{m})$ is written as $\gras(r,m)$. We consider
the Radon transform of a given ``function'' $f$ on $M_{1}$. Roughly
speaking, it is a process of obtaining a ``function'' $\cR f$ on
$M_{2}$. Here ``function'' in our setting is actually a section
of a certain line bundle on $M_{1}$ or $M_{2}$. Let $\pi_{1}^{*}f$
be the pullback of $f$ by $\pi_{1}$; take any $v\in M_{2}$ and
consider the restriction of $\pi_{1}^{*}f$ to the fiber $\pi_{2}^{-1}(v)$;
one obtains a form $\pi_{1}^{*}f|_{\pi_{2}^{-1}(v)}\cdot\tau$ on
$\pi_{2}^{-1}(v)$ with an appropriately chosen form $\tau$ of degree
$\dim\pi_{2}^{-1}(v)$; integrate it on some chain $C(v)$ in $\pi_{2}^{-1}(v)$
of dimension $\dim\pi_{2}^{-1}(v)$; one obtains $\cR f:M_{2}\ni v\mapsto\int_{C(v)}\pi_{1}^{*}f|_{\pi_{2}^{-1}(v)}\cdot\tau$.
Note that $\pi_{2}^{-1}(v)=\{(v_{1},v)\in\fl(r,m,V)\}\simeq\gras(r,v)\simeq\gras(r,m)$
and hence, $\dim\pi_{2}^{-1}(v)=r(m-r)$ and $\tau$ is an $r(m-r)$-form. 

Let us describe the above process more explicitly using the coordinates
of Grassmannian. Let $\mnm$ denote the set of $m\times N$ complex
matrices of maximum rank. We identify $\gras(m,V)$ with the quotient
space $\GL m\backslash\mnm$ of $\mat'(m,N)$ by the action of $\GL m$,
where $\GL m$ acts on $\mnm$ from the left by matrix multiplication.
Then the identification is given by 
\[
\GL m\backslash\mnm\ni[z]\mapsto v=\Span\{z_{1}',\dots,z_{m}'\}\in\gras(m,V),
\]
where $z_{i}'$ is the $i$-th row vector of $z\in\mat'(m,N)$ and
$z_{1}',\dots,z_{m}'$ spans an $m$-dimesional subspace of $V$.
It is easy to see that $v$ is independent of the choice of representative
$z$ of a class $[z]$. In this case, $\mat'(m,N)$ is called the
space of homogeneous coordinates of $\gras(m,V)$ in this paper.

Similarly, the identification of $\pi_{2}^{-1}(v)=\gras(r,v)$ with
$\GL r\backslash\mat'(r,m)$ is given by 

\[
\GL r\backslash\mat'(r,m)\ni[t]\mapsto\Span_{\C}\{t_{1}'z,\dots,t_{r}'z\}\in\gras(r,v),
\]
where $t_{i}'\in\C^{m}$ is the $i$-th row vector of $t$ and $t_{i}'z=\sum_{k=1}^{m}t_{i,k}z_{k}'$
is a vector belonging to the subspace $v\subset V$. So, if $f$ is
given as a function on the space $\mat'(r,N)$ of homogeneous coordinates
of $M_{1}$, we see that $\pi_{1}^{*}f|_{\pi_{2}^{-1}(v)}$ is expressed
as $f(tz)$. Thus a Radon transform of $f$ is written as 
\[
\GL m\backslash\mnm\ni[z]\mapsto(\cR f)(z):=\int f(tz)\cdot\tau(t).
\]
About the form $\tau(t)$, we explain in the definition of Radon HGF
in the next section.

Next step is to prepare a function $f$ to be integrated to define
the Radon HGF. We use a character of some Lie subgroup of $\GL N$.
For this end we assume the following.
\begin{assumption}
$N=rn$ for some positive integer $n$.
\end{assumption}

\subsection{\label{subsec:Radon-HGF-nonconf}Radon HGF of non-confluent type}

We consider the subgroup of $G=\GL N$: 
\[
H:=\left\{ h=\left(\begin{array}{ccc}
h_{1}\\
 & \ddots\\
 &  & h_{n}
\end{array}\right)\mid h_{i}\in\GL r\right\} \simeq\GL r{}^{n}.
\]
When $r=1$, $H$ reduces to a Cartan subgroup of $G$. The following
lemma gives a character of the universal covering group of $H$.
\begin{lem}
\label{lem:char-nonconf}Let $\tilde{H}$ be the universal covering
group of $H$. Then any character $\chi:\tilde{H}\to\cbatu$ is given
by 
\begin{equation}
\chi(h;\al)=\prod_{i=1}^{n}(\det h_{i})^{\al_{i}},\qquad h=\diag(h_{1},\dots,h_{n})\in\tilde{H}\label{eq:char-0}
\end{equation}
for some $\al=(\al_{1},\dots,\al_{n})\in\C^{n}$. $\al$ is said to
be a weight of $\chi$.
\end{lem}

\begin{proof}
Since $H\simeq\left(\GL r\right)^{n}$, it is sufficient to show that
a character $f:\widetilde{\GL r}\to\cbatu$ is given by $f(x)=(\det x)^{\al}$
for some $\al\in\C$. Let $\gl(r)$ be the Lie algebra of $\GL r$
and let $\phi$ be the differential of $f$, i.e. the Lie algebra
homomorphism $\phi=df:\gl(r)\to\C$ associated with $f$. So $\phi$
is a linear map satisfying
\begin{equation}
\phi(gXg^{-1})=\phi(X),\quad\forall X\in\gl(r),\forall g\in\GL r.\label{eq:nonconfl-1}
\end{equation}
We assert that $\phi(X)=\al\Tr(X)$ for some $\al\in\C$. Let $X\in\gl(r)$
be a generic element such that its eigenvalues are all distinct. Then
$X=g\cdot\diag(x_{1},\dots,x_{r})\cdot g^{-1}$ with some $g\in\GL r$.
Then 
\begin{equation}
\phi(X)=\phi\left(g\cdot\diag(x_{1},\dots,x_{r})\cdot g^{-1}\right)=\phi(\diag(x_{1},\dots,x_{r}))=\sum_{i=1}^{r}\al_{i}x_{i}\label{eq:char-0-1}
\end{equation}
for some constants $\al_{1},\dots,\al_{r}\in\C$. We show that $\al_{1}=\cdots=\al_{r}$.
Now in (\ref{eq:nonconfl-1}) we take $X=\diag(x_{1},\dots,x_{r})$
and $g$ the permutation matrix corresponding to a permutation $\sigma\in\Si_{r}$.
Then from (\ref{eq:nonconfl-1}) and (\ref{eq:char-0-1}) we have
\[
\sum_{i=1}^{r}\al_{i}x_{i}=\sum_{i=1}^{r}\al_{\sigma^{-1}(i)}x_{i},
\]
which holds for any $\sigma\in\Si_{r}$ and any $x_{1},\dots,x_{r}\in\C$.
It follows that all $\al_{i}$ are equal which we denote as $\al$.
Thus we have $\phi(X)=\sum_{i}\al x_{i}=\al\Tr(X)$. Now chasing the
diagram
\[
\begin{CD}\widetilde{\GL r}@>f>>\cbatu\\
@V\log VV@AA\exp A\\
\gl(r)@>\phi:=df>>\C,
\end{CD}
\]
we have $f(x)=(\exp\circ\phi\circ\log))(x)=\exp(\phi(X))=\exp(\phi(X))=\exp\left(\al\Tr(X)\right)=(\det x)^{\al}$.
\end{proof}
We want to regard the character $\chi(\cdot;\al)$ given by (\ref{eq:char-0})
as a multivalued section of a certain line bundle on $M_{1}=\gras(r,N)\simeq\GL r\backslash\mat'(r,N)$.
For this, consider the map $\iota:H\to\mat'(r,N)$ defined by 
\[
H\ni\diag(h_{1},\dots,h_{n})\mapsto(h_{1},\dots,h_{n})\in\mat'(r,N),
\]
and identify $H$ as a Zariski open subset $\iota(H)\subset\mat'(r,N)$.
Then $\chi(h;\al)$ is regarded as a multivalued analytic function
on this Zariski open subset. Through the map $\iota$, a multiplication
$g\cdot h=\diag(g,\dots,g)\cdot\diag(h_{1},\dots,h_{n})$ in $H$
is translated as an action of $\GL r$ on $\mat'(r,N)$ given by $\GL r\times\mat'(r,N)\ni(g,h)\mapsto gh=(gh_{1},\dots,gh_{n})\in\mat'(r,N)$.
Then we have 
\begin{equation}
\chi(gh;\al)=\prod_{i=1}^{n}(\det(gh_{i}))^{\al_{i}}=(\det g)^{\al_{1}+\cdots+\al_{n}}\chi(h;\al).\label{eq:char-1}
\end{equation}

\begin{assumption}
\label{assu:radon-nonconf}The weight $\al$ of $\chi(\cdot;\al)$
satisfies 

\[
\al_{1},\dots,\al_{n}\notin\Z\quad\text{and }\quad\al_{1}+\cdots+\al_{n}=-m.
\]
\end{assumption}

Let $\rho_{m}:\GL r\to\cbatu$ be a character defined by $\rho_{m}(g)=(\det g)^{m}$,
then (\ref{eq:char-1}) can be written as
\begin{equation}
\chi(gh;\al)=\rho_{m}(g)^{-1}\chi(h;\al)\label{eq:char-2}
\end{equation}
and it implies that $\chi(\cdot;\al)$ can be regarded as a multivalued
analytic section of the line bundle on $\gras(r,V)\simeq\GL r\backslash\mat'(r,N)$
determined by the character $\rho_{m}$. This line bundle is denoted
as $L_{\rho_{m}}$. Let us explain the meaning of the assumption $\al_{1}+\cdots+\al_{n}=-m$.
Put $T=\gras(r,m)=\GL r\backslash\mat'(r,m)$ and let $K_{T}$ be
the canonical bundle of $T$ whose holomorphic sections are holomorphic
forms of the highest degree ($=\dim_{\C}T=r(m-r))$. The sheaf of
holomorphic sections of $K_{T}$ is denoted as $\cO(K_{T})$. This
sheaf is described as follows. Let $p:T\to\Ps(\wedge^{r}\C^{m})$
be the Pl\"ucker embedding given by 
\[
T\ni[t]\mapsto[(p_{J})_{J}]\in\Ps(\wedge^{r}\C^{m}),
\]
where $J$ runs on the set $\{J=(j_{1},\dots,j_{r})\mid1\leq j_{1}<\cdots<j_{r}\leq m\}$
and $p_{J}:=\det(t_{j_{1}},\dots,t_{j_{r}})$ is the $J$-th Pl\"ucker
coordinate for $t=(t_{1},\dots,t_{m})$. Then it is known that $\cO(K_{T})\simeq p^{*}\cO(-m$),
where $\cO(-m)$ is a sheaf of local sections of the line bundle $(H^{\otimes m})^{*}$,
$H$ being the hyperplane bundle on $\Ps(\wedge^{r}\C^{m})$, see
\cite{Brion,Griffith}. This means that in each local chart $U_{J}=\{[t]\in T\mid p_{J}\neq0\}$,
there is an element $\omega_{J}\in\W^{r(m-r)}(U_{J})$, which is a
basis of $\cO(U_{J})$ module $\W^{r(m-r)}(U_{J})$, such that $\wm_{J}=(p_{J'}/p_{J})^{m}\wm_{J'}$
on $U_{J}\cap U_{J'}$. Hence $p_{J}^{m}\wm_{J}=p_{J'}^{m}\wm_{J'}$
on $U_{J}\cap U_{J'}$. Put $\tau(t):=p_{J}^{m}\wm_{J}$ on $U_{J}$.
Noting $p_{J}^{m}=(p_{J}/p_{J'})^{m}p_{J'}^{m}$, the form $\tau$
is regarded as an element of $p^{*}\cO(-m)\otimes p^{*}\cO(m)=p^{*}\cO(0)$.
Hence $\tau(t)$ is determined uniquely up to constant multiple. For
$J_{0}:=(1,\dots,r)$, we take $\wm_{J_{0}}$ as follows. Let the
homogeneous coordinates $t$ of $T$ be written as $t=(t',t'')$ with
$t'\in\mat(r),t''\in\mat(r,m-r)$. Let the affine coordinates $u$
of $U_{J_{0}}$ be defined as $u=(t')^{-1}t''$. Then we take $\wm_{J_{0}}=\wedge_{i,j}du_{i,j}$,
which will be denoted as $du$. Hence in $U_{J_{0}}$, we can write
$\tau(t)$ as 
\begin{equation}
\tau(t)=(\det t')^{m}du.\label{eq:char-2-1}
\end{equation}

\begin{example*}
In the case $T=\gras(1,m)=\Ps^{m-1}$ with the homogeneous coordinates
$t=(t_{1},\dots,t_{m})$, we can take $\tau=\sum_{1\leq j\leq m}(-1)^{m+1}t_{j}dt_{1}\wedge\cdots\widehat{dt_{j}}\cdots\wedge dt_{m}$.
Then in the neibourghhood $U_{1}=\{[t]\in T\mid t_{1}\neq0\}$ with
the affine coordinates $(u_{2},\dots,u_{m})=(t_{2}/t_{1},\dots,t_{m}/t_{1})$,
we have 
\[
\tau=t_{1}^{m}d\left(\frac{t_{2}}{t_{1}}\right)\wedge\cdots\wedge d\left(\frac{t_{m}}{t_{1}}\right)=t_{1}^{m}du_{2}\wedge\cdots\wedge du_{m}.
\]
\end{example*}
We also see that 
\begin{equation}
\tau(gt)=\rho_{m}(g)\tau(t),\quad g\in\GL r\label{eq:char-3}
\end{equation}
holds. Noting $N=nr$, put 
\[
Z=\{z=(z_{1},\dots,z_{n})\in\mat'(m,N)\mid z_{j}\in\mat(m,r),\rank z_{j}=r\;(1\leq j\leq n)\}
\]
which is a Zariski open subset of $\mat'(m,N)$. Then for $z\in Z$,
$\chi(tz;\al)=\prod_{1\leq j\leq n}\left(\det(tz_{j})\right)^{\al_{j}}$
satisfies 
\begin{equation}
\chi((gt)z;\al)=(\det g)^{-m}\chi(tz;\al)\label{eq:char-4}
\end{equation}
by the assumption, which means that $\chi(tz;\al)$ gives a multivalued
section of the line bundle $L_{\rho_{m}}$ on $T$ and that the branch
locus is $\cup_{j}S_{z}^{(j)}\subset T$, where $S_{z}^{(j)}:=\{[t]\in T\mid\det(tz_{j})=0\}$.
Hence by (\ref{eq:char-3}) and (\ref{eq:char-4}), we see that the
form $\chi(tz;\al)\cdot\tau(t)$ is a multivalued analytic $r(m-r$)-form
on $X_{z}:=T\setminus\left(\cup_{j}S_{z}^{(j)}\right)$. Let $\cL_{z}$
be the local system on $X_{z}$ defined by the monodromy of $\chi(tz;\al)\cdot\tau(t)$.
\begin{defn}
The integral 
\begin{equation}
F(z,\al;C)=\int_{C(z)}\chi(tz;\al)\cdot\tau(t),\quad z\in Z\label{eq:nonconfl-2}
\end{equation}
is called the \emph{Radon hypergeometric integral} (Radon HGI), where
$C(z)$ is a cycle of the homology group $H_{r(m-r)}^{lf}(X_{z};\cL_{z})$
of locally finite chains with coefficients in the local system $\cL_{z}$.
\end{defn}

To define $F(z,\al;C)$ as a function of $z$, which we call the Radon
hypergeometric function (Radon HGF), we must take $C(z)$ depending
continuously on $z$. Put 
\[
\mathscr{X}:=\cup_{z\in Z}X_{z}=\{([t],z)\mid\det(tz_{j})\neq0,\;1\leq j\leq n\}\subset T\times Z.
\]
It is a Zariski open subset of a smooth algebraic variety $T\times Z$.
Then the projection $\pi_{2}:T\times Z\to Z$ to the second factor
$([t],z)\mapsto z$ induces a fibration $\pi:\mathscr{X}\to Z$, where
$\pi=\pi_{2}|_{\mathscr{X}}$. Note that $\mathscr{X}$ and $Z$ are
smooth algebraic varieties and $\pi$ is a surjective morphism. To
this setting, we apply the following result which can be found as
Corollary 5.1 in \cite{Verdier}.
\begin{prop}
\label{prop:Verdier} Let $X,Y$ be nonsingular algebraic varieties
of finite type and $f:X\to Y$ be a morphism. Then there exists a
Zariski open subset $W$ of $Y$ such that $f:f^{-1}(W)\to W$ is
a locally trivial topological fibration. Namely, for any point $w\in W$,
there exits an open neighbourhood $U$ (in the transcendental topology)
such that there exist a homeomorphism $g:f^{-1}(U)\to U\times f^{-1}(w)$
satisfying $\pi_{1}\circ g=f$, where $\pi_{1}$ is the projection
to the first factor. 
\end{prop}

Applying the proposition to our setting $\pi:\mathscr{X}\to Z$, there
is a Zariski open subset $W\subset Z$ such that $\pi:\cup_{z\in W}X_{z}\to W$
is a locally trivial fibration. It follows that we have a local system
of vector space on $W$:
\[
\bigcup_{z\in W}H_{r(m-r)}^{lf}(X_{z};\cL_{z})\to W.
\]
Then we take a continuous local section $C=\{C(z)\}$ of this flat
vector bundle. This choice assure the analyticity of $F(z,\al;C)$
on $z\in W$ for an appropriate $\al$.

We can write the integral (\ref{eq:nonconfl-2}) in terms of the coordinates
$u=(u_{i,j})\in\mat(r,m-r)$ in an affine chart $U_{J_{0}}$ of $T$.
Using (\ref{eq:char-2-1}) and (\ref{eq:char-4}), $F(z,\al;C)$ is
expressed as 
\begin{equation}
F(z,\al;C)=\int_{C(z)}\chi(\vec{u}z;\al)du=\int_{C(z)}\prod_{1\leq j\leq n}\left(\det(\vec{u}z_{j})\right)^{\al_{j}}du,\quad z=(z_{1},\dots,z_{n}),\label{eq:nonconfl-3}
\end{equation}
where $\vec{u}=(1_{r},u)$.

\subsection{Radon HGF of confluent type}

\subsubsection{Character of Jordan group}

To define the Radon HGF of confluent type, we use, in place of the
group $H\simeq\GL r^{n}$, an analogue of direct product of Jordan
groups. Let $\lm=(n_{1},\dots,n_{\ell})$ be a partition of $n$,
namely $\lm$ is a nonincreasing sequence of positive integers such
that $|\lm|:=n_{1}+\cdots+n_{\ell}=n$. For $\lm$, let $H_{\lm}$
be a subgroup of $G=\GL N$ defined by
\[
H_{\lm}=J_{r}(n_{1})\times\cdots\times J_{r}(n_{\ell}),
\]
where
\[
J_{r}(p):=\left\{ h=\left(\begin{array}{cccc}
h_{0} & h_{1} & \dots & h_{p-1}\\
 & \ddots & \ddots & \vdots\\
 &  & \ddots & h_{1}\\
 &  &  & h_{0}
\end{array}\right)\mid h_{0}\in\GL r,\ h_{i}\in\mat(r)\right\} \subset\GL{pr}
\]
is a Lie group called (generalized) Jordan group. An element $(h^{(1)},\dots,h^{(\ell)})\in J_{r}(n_{1})\times\cdots\times J_{r}(n_{\ell})$
is identified with a block diagonal matrix $\diag(h^{(1)},\dots,h^{(\ell)})\in\GL N$.
We also introduce a unipotent group $\jro(p)\subset J_{r}(p)$ by
\[
\jro(p):=\left\{ h=\left(\begin{array}{cccc}
1_{r} & h_{1} & \dots & h_{p-1}\\
 & \ddots & \ddots & \vdots\\
 &  & \ddots & h_{1}\\
 &  &  & 1_{r}
\end{array}\right)\mid h_{i}\in\mat(r)\right\} .
\]
An element $h\in J_{r}(p)$ is expressed as $h=\sum_{0\leq i<p}h_{i}\otimes\La^{i}$
using the shift matrix $\La=(\de_{i+1,j})$ of size $p$. Let us describe
the characters of the Jordan group. The following lemma is easily
shown.
\begin{lem}
\label{lem:radon-conf-1}We have a group isomorphism 
\[
J_{r}(p)\simeq\GL r\ltimes\jro(p)
\]
defined by the correspondence $\GL r\ltimes\jro(p)\ni(g,h)\mapsto g\cdot h=\sum_{0\leq i<p}(gh_{i})\otimes\La^{i}\in J_{r}(p)$,
where the semi-direct product is defined by the action of $\GL r$
on $\jro(p)$: $h\mapsto g^{-1}hg=\sum_{i}(g^{-1}h_{i}g)\otimes\La^{i}$.
\end{lem}

Taking into account the expression $h=\sum_{0\leq i<p}h_{i}\otimes\La^{i}$,
we can describe the Jordan group $J_{r}(p)$ as follows. Put $R=\mat(r)$
and consider it as a $\C$-algebra. Let $R[w]$ be the ring of polynomials
in $w$ with coefficients in $R$. Then $J_{r}(p)$ is identified
with the group of units in the quotient ring of $R[w]$ by the principal
ideal $(w^{p})$:
\[
J_{r}(p)\simeq\left(R[w]/(w^{p})\right)^{\times}.
\]
Thus we can write $J_{r}(p)$ and $\jro(p)$ as 
\begin{align*}
J_{r}(p) & \simeq\left\{ h_{0}+\sum_{1\leq i<p}h_{i}w^{i}\in R[w]/(w^{p})\mid h_{0}\in\GL r\right\} ,\\
\jro(p) & \simeq\left\{ 1_{r}+\sum_{1\leq i<p}h_{i}w^{i}\in R[w]/(w^{p})\right\} .
\end{align*}
In the following we use freely this identification when necessary.

Let us determine the characters of the universal covering group $\tilde{J}_{r}(p)$
of $J_{r}(p)$. Since $J_{r}(p)\simeq\GL r\ltimes\jro(p)$ and characters
of $\widetilde{\GL r}$ are already determined in the proof of Lemma
\ref{lem:char-nonconf}, it is sufficient to determine the characters
of $\jro(p)$ which come from those of $\tilde{J}_{r}(p)$ by restriction.
Let $\jroo(p)$ be the Lie algebra of $\jro(p)$:
\[
\jroo(p)=\{X=\sum_{1\leq i<p}X_{i}w^{i}\mid X_{i}\in R\},
\]
where the Lie bracket of $X,Y\in\jroo(p)$ is given by $[X,Y]=\sum_{2\leq k<p}\sum_{i+j=k}[X_{i},Y_{j}]w^{k}$.
Note that we have an isomorphism $\jroo(p)\simeq\bigoplus_{1\leq i<p}Rw^{i}\simeq\oplus^{r-1}R$
as a vector space. To obtain a character of $\jro(p)$, we lift the
character of $\jroo(p)$ to that of $\jro(p)$ by the exponential
map. Since $\jro(p)$ is a simply connected Lie group, the exponential
map
\[
\exp:\jroo(p)\to\jro(p),\,X\mapsto\exp(X)=\sum_{k=0}^{\infty}\frac{1}{k!}X^{k}=\sum_{0\leq k<p}\frac{1}{k!}X^{k}
\]
is a biholomorphic map. Hence we can consider the inverse map $\log:\jro(p)\to\jroo(p)$,
which, for $h=1_{r}+\sum_{1\leq i<p}h_{i}w^{i}\in\jro(p)$, defines
$\te_{k}(h)\in R$ by
\begin{align*}
\log h & =\log\left(1_{r}+h_{1}w+\cdots+h_{p-1}w^{p-1}\right)\\
 & =\sum_{1\leq k<p}\frac{(-1)^{k+1}}{k}\left(h_{1}w+\cdots+h_{p-1}w^{p-1}\right)^{k}\\
 & =\sum_{1\leq k<p}\theta_{k}(h)w^{k}.
\end{align*}
Here $\te_{k}(h)$ is given as a sum of monomials of noncommutative
elements $h_{1},\dots,h_{p-1}\in R$. If a weight of $h_{i}$ is defined
to be $i$, then the monomials appearing in $\te_{k}(h)$ has the
weight $k$. For example we have 
\begin{align*}
\te_{1}(h) & =h_{1},\\
\te_{2}(h) & =h_{2}-\frac{1}{2}h_{1}^{2},\\
\te_{3}(h) & =h_{3}-\frac{1}{2}(h_{1}h_{2}+h_{2}h_{1})+\frac{1}{3}h_{1}^{3},\\
\te_{4}(h) & =h_{4}-\frac{1}{2}(h_{1}h_{3}+h_{2}^{2}+h_{3}h_{1})+\frac{1}{3}(h_{1}^{2}h_{2}+h_{1}h_{2}h_{1}+h_{2}h_{1}^{2})-\frac{1}{4}h_{1}^{4}.
\end{align*}

\begin{lem}
\label{lem:Radon-conf-2}Let $\chi:\jro(p)\to\cbatu$ be a character
obtained from that of $\tilde{J}_{r}(p)$ by restricting it to $\jro(p)$.
Then there exists $\al=(\al_{1},\dots,\al_{p-1})\in\C^{p-1}$ such
that 
\begin{equation}
\chi(h;\al)=\exp\left(\sum_{1\leq i<p}\al_{i}\Tr\,\theta_{i}(h)\right).\label{eq:char-conf-1}
\end{equation}
Conversely, $\chi$ defined by (\ref{eq:char-conf-1}) gives a character
of $\jro(p)$.
\end{lem}

\begin{proof}
A character $\chi$ is obtained by tracing the commutative diagram:
\[
\begin{CD}\jro(p)@>\chi>>\cbatu\\
@A\exp AA@AA\exp A\\
\jroo(p)@>\phi:=d\chi>>\C.
\end{CD}
\]
Note that $\GL r$ acts on $\jroo$ by $X=\sum_{1\leq i<p}X_{i}w^{i}\to gXg^{-1}:=\sum_{1\leq i<p}(gX_{i}g^{-1})w^{i}$.
Then $\phi$ is a linear map satisfying $\phi(gXg^{-1})=\phi(X)$
for any $g\in\GL r$ since $\chi$ is a restriction of a character
of $\tilde{J}_{r}(p)$. Since $\phi$ is linear, we see that 
\[
\phi(X)=\phi(\sum_{1\leq i<p}X_{i}w^{i})=\sum_{1\leq i<p}\phi(X_{i}w^{i})=:\sum_{1\leq i<p}\phi_{i}(X_{i}),
\]
where $\phi_{i}:\mat(r)\to\C$ is a linear map obtained by restricting
$\phi$ to $Rw^{i}$. Note that we have $\phi_{i}(gYg^{-1})=\phi_{i}(Y)$
for any $Y\in R$. By the fact shown in the proof of Lemma \ref{lem:char-nonconf},
we have $\phi_{i}(X_{i})=\al_{i}\Tr\,X_{i}$ for some $\al_{i}\in\C$.
Hence $\phi(X)=\sum_{1\leq i<p}\al_{i}\Tr\,X_{i}$. Now from the diagram
we have 
\[
\chi(h)=(\exp\circ\phi\circ(\exp)^{-1})(h)=(\exp\circ\phi)\left(\sum_{1\leq i<p}\te_{i}(h)w^{i}\right)=\exp\left(\sum_{1\leq i<p}\al_{i}\Tr\,\te_{i}(h)\right).
\]
\end{proof}
By virtue of the isomorphism in Lemma \ref{lem:radon-conf-1}, we
can determine the characters of $\tilde{J}_{r}(p)$ as a corollary
of Lemma \ref{lem:Radon-conf-2}.
\begin{cor}
\label{cor:Radon-conf-3}Any character $\chi_{p}:\tilde{J}_{r}(p)\to\cbatu$
is given by 
\[
\chi_{p}(h;\al)=(\det h_{0})^{\al_{0}}\exp\left(\sum_{1\leq i<p}\al_{i}\Tr\,\theta_{i}(\underline{h})\right),
\]
for some $\al=(\al_{0},\al_{1},\dots,a_{p-1})\in\C^{p}$, where $\underline{h}\in\jro(p)$
is defined by $h=\sum_{0\leq i<p}h_{i}w^{i}=\sum_{0\leq i<p}h_{0}(h_{0}^{-1}h_{i})w^{i}=h_{0}\cdot\underline{h}$.
\end{cor}

Now the characters of the group $\tilde{H}_{\lm}$ are given as follows.
\begin{prop}
\label{prop:char-conf}For a character $\chi_{\lm}:\tH_{\lm}\to\cbatu$,
there exists $\alpha=(\alpha^{(1)},\dots,\alpha^{(\ell)})\in\C^{N}$,
$\alpha^{(k)}=(\alpha_{0}^{(k)},\alpha_{1}^{(k)},\dots,\alpha_{n_{k}-1}^{(k)})\in\C^{n_{k}}$
such that 
\[
\chi_{\lm}(h;\al)=\prod_{1\leq k\leq\ell}\chi_{n_{k}}(h^{(k)};\al^{(k)}),\quad h=(h^{(1)},\cdots,h^{(\ell)})\in\tilde{H}_{\lm},\;h^{(k)}\in\tilde{J}_{r}(n_{k}).
\]
\end{prop}

\subsubsection{Definition of HGF of confluent type }

To consider the Radon transform of a character $\chi_{\lm}$, we regard
$\chi_{\lm}$ as an object on the Grassmannian $M_{1}:=\gras(r,N)$.
So, as is done in the non-confluent case, we define at first an injective
map $\iota_{\lm}:H_{\lm}\to\mat'(r,N)$ by 
\begin{equation}
h=(h^{(1)},\dots,h^{(\ell)})\mapsto(h_{0}^{(1)},\dots,h_{n_{1}-1}^{(1)},\dots,h_{0}^{(\ell)},\dots,h_{n_{\ell}-1}^{(\ell)}),\label{eq:radon-conf-3}
\end{equation}
where $h^{(j)}=\sum_{0\leq k<n_{j}}h_{k}^{(j)}w^{k}$ with $h_{k}^{(j)}\in\mat(r)$
and the right hand side is the matrix obtained by arraying horizontally
the $r\times r$ matrices in the parenthesis. By the map $\iota_{\lm}$,
we identify $H_{\lm}$ as a Zariski open subset $\iota_{\lm}(H_{\lm})$
of $\mat'(r,N)$ and regard the character $\chi_{\lm}(\cdot;\al)$
as a multivalued analytic  function defined on this open subset. Next,
to regard $\chi_{\lm}$ as an object on the Grassmannian, consider
the action $\GL r\curvearrowright\mat'(r,N):(g,h)\mapsto gh$. Then
$\chi_{\lm}(\cdot;\al)$ behaves as 
\begin{align}
\chi_{\lm}(gh;\al) & =\left(\det g\right)^{\al_{0}^{(1)}+\cdots+\al_{0}^{(\ell)}}\prod_{j=1}^{\ell}\det(h_{0}^{(j)})^{\al_{0}^{(j)}}\exp\left(\sum_{1\leq k<n_{j}}\al_{k}^{(j)}\Tr\,\theta_{k}(\underline{h^{(j)}})\right)\nonumber \\
 & =\left(\det g\right)^{\al_{0}^{(1)}+\cdots+\al_{0}^{(\ell)}}\chi_{\lm}(h;\al).\label{eq:radon-conf-4}
\end{align}
Here we used the fact that $\underline{h^{(j)}}=(h_{0}^{(j)})^{-1}h^{(j)}=\sum_{0\leq k<n_{j}}(h_{0}^{(j)})^{-1}h_{k}^{(j)}w^{k}$
is invariant for the change $h\mapsto gh$. Taking this property into
account, we set the following assumption.
\begin{assumption}
\label{assu:Radon-conf-4}(i) $\al_{0}^{(j)}\notin\Z$ for $1\leq j\leq\ell$,

(ii) $\al_{n_{j}-1}^{(j)}\neq0$ if $n_{j}\geq2$,

(iii) $\al_{0}^{(1)}+\cdots+\al_{0}^{(\ell)}=-m$.
\end{assumption}

Let $\rho_{m}:\GL r\to\cbatu$ be a character defined by $\rho_{m}(g)=(\det g)^{m}$,
then (\ref{eq:radon-conf-4}) can be written as
\begin{equation}
\chi_{\lm}(gh;\al)=\rho_{m}(g)^{-1}\chi_{\lm}(h;\al)\label{eq:char-2-2}
\end{equation}
and it implies that $\chi_{\lm}(\cdot;\al)$ can be regarded as a
multivalued analytic section of the line bundle $L_{\rho_{m}}$ on
$\gras(r,V)\simeq\GL r\backslash\mat'(r,N)$ determined by the character
$\rho_{m}$.

The meaning of the assumption (iii) is the same as in the non-confluent
case. Put $T=\gras(r,m)=\GL r\setminus\mat'(r,m)$ as in Section \ref{subsec:Radon-HGF-nonconf}.
According as the partition $\lm$ of $n$, we write $z\in\mat'(m,N)$
as
\[
z=(z^{(1)},\dots,z^{(\ell)}),\;z^{(j)}=(z_{0}^{(j)},\dots,z_{n_{j}-1}^{(j)}),\;z_{k}^{(j)}\in\mat(m,r).
\]
Consider 
\[
Z=\{z=(z^{(1)},\dots,z^{(\ell)})\in\mat'(m,N)\mid\rank z_{0}^{(k)}=r\;(1\leq k\leq\ell)\},
\]
which is a Zariski open subset of $\mat'(m,N)$. For $z\in Z$, $\chi_{\lm}(tz;\al)$
satisfies 
\begin{equation}
\chi_{\lm}((gt)z;\al)=\rho_{m}(g)^{-1}\chi_{\lm}(tz;\al)\label{eq:radon-conf-5}
\end{equation}
by virtue of (\ref{eq:char-2-2}). This implies that $\chi_{\lm}(tz;\al)$
gives a multivalued section of the line bundle on $T$ determined
by the character $\rho_{m}$ with the branch locus $\cup_{1\leq j\leq\ell}S_{z}^{(j)}$,
where $S_{z}^{(j)}:=\{[t]\in T\mid\det(tz_{0}^{(j)})=0\}$ is a hypersurface
in $T$ of degree $r$. Put $X_{z}:=T\setminus\cup_{1\leq j\leq\ell}S_{z}^{(j)}$,
which is the complement of the arrangement of hypersurfaces $\{S_{z}^{(1)},\dots,S_{z}^{(\ell)}\}$
in $T$. By virtue of (\ref{eq:char-3}) and (\ref{eq:radon-conf-5}),
$\chi_{\lm}(tz;\al)\cdot\tau(t)$ gives a multivalued $r(m-r)$-form
on $X_{z}$. The monodromy of this form defines a local system $\cL_{z}$
of rank $1$ on $X_{z}$.
\begin{defn}
For a character $\chi_{\lm}(\cdot;\al)$ of the group $\tH_{\lm}$
satisfying Assumption \ref{assu:Radon-conf-4}, 
\begin{equation}
F_{\lm}(z,\al;C)=\int_{C(z)}\chi_{\lm}(tz;\al)\cdot\tau(t)\label{eq:confl-1}
\end{equation}
is called the Radon HGI of type $\lm$. $C(z)$ is an $r(m-r)$-cycle
of the homology group of locally finite chains $H_{r(m-r)}^{\Phi_{z}}(X_{z};\cL_{z})$
of $X_{z}$ with coefficients in the local system $\cL_{z}$ and with
the family of supports $\Phi_{z}$ determined by $\chi_{\lm}(tz;\al)$. 
\end{defn}

For this homology group and a family of supports, we discuss in the
next section.

As in the non-confluent case, we can give an expression of $F_{\lm}$
in terms of affine coordinates $u=(u_{i,j})=(t')^{-1}t''$ of an affine
chart $U_{J_{0}}=\{[t]\in T\mid\det t'\neq0\}$, where $t=(t',t'')$
is the homogeneous coordinates of $T$. Using (\ref{eq:char-3}) and
(\ref{eq:radon-conf-5}), we have 
\begin{align*}
F_{\lm}(z,\al;C) & =\int_{C(z)}\chi_{\lm}(\vec{u}z;\al)du,\\
 & =\int_{C(z)}\prod_{j=1}^{\ell}\left(\det(\vec{u}z_{0}^{(j)})\right)^{\al_{0}^{(j)}}\exp\left(\sum_{1\leq j\leq n}\sum_{1\leq k<n_{j}}\al_{k}^{(j)}\Tr\,\theta_{k}(\underline{\vec{u}z}^{(k)})\right)du
\end{align*}
where $\vec{u}=(1_{r},u)$.

\subsubsection{Cycle $C(z)$ in the integral }

First we consider the situation where $z$ is fixed. Write the integrand
of (\ref{eq:confl-1}) as 
\[
\chi_{\lm}(tz;\al)=f(t,z)\exp(g(t,z)),
\]
where
\[
f(t,z)=\prod_{j=1}^{\ell}\left(\det(tz_{0}^{(j)})\right)^{\al_{0}^{(j)}},\quad g(t,z)=\sum_{1\leq j\leq n}\sum_{1\leq k<n_{j}}\al_{k}^{(j)}\Tr\,\theta_{k}(\underline{tz}^{(k)})
\]
Note that $f(t,z)\cdot\tau(t)$ concerns the multivalued nature of
the integrand whose ramification locus is $\cup_{j}S_{z}^{(j)}$.
On the other hand, $g(t,z)$ is a rational function on $T$ with a
pole divisor $\cup_{j;n_{j}\geq2}S_{z}^{(j)}$ and concerns the nature
of exponential increase to infinity or exponential decrease to zero
of the integrand when $[t]$ approaches to the pole divisor $\cup_{j;n_{j}\geq2}S_{z}^{(j)}$.
The local system $\cL_{z}$ on $X_{z}$ is the same as that determined
by the monodromy of $f(t,z)\cdot\tau(t)$. The information of $g(t,z)$
is described using the notion of family of supports. Let $\Phi_{z}$
be a family of closed subsets of $X_{z}$ satisfying the following
condition: 
\[
A\in\Phi_{z}\iff A\cap g_{z}^{-1}(\{w\in\C\mid\mathrm{Re}\,w\geq a\})\ \ \mbox{is compact for any \ensuremath{a\in\R}},
\]
where $g(t,z)$, with $z$ fixed, is considered as a map $g_{z}:=g|_{X_{z}}:X_{z}\to\C$.
Then the family $\Phi_{z}$ enjoys the following properties:
\begin{enumerate}
\item If $A,B\in\Phi_{z}$, then $A\cup B\in\Phi_{z}.$
\item If $A\in\Phi_{z}$ and $B$ is a closed subset of $A,$ then $B\in\Phi_{z}.$
\item For any $A\in\Phi_{z}$ there exists $B\in\Phi_{z}$ whose interior
$B^{\circ}$ is a neighbourhood of $A.$ 
\end{enumerate}
In general, a family $\Phi$ of closed subsets of a topological space
$X$ is said to be a family of supports when $\Phi$ satisfies the
above conditions $1,2$ and $3$. Then we can define the homology
group $H_{k}^{\Phi_{z}}(X_{z};\cL_{z})$ of locally finite chains
with coefficients in the local system $\cL_{z}$ and with the family
of support $\Phi_{z}$, see \cite{Pham-1} for the details. Applying
Proposition \ref{prop:Verdier} to the map $g_{z}:X_{z}\to\C$, we
see that this map gives a locally trivial fibration on $\C\setminus D$
for some finite set $D\subset\C$. For any $c\in\R$, put
\[
B_{c}:=g_{z}^{-1}(\{w\in\C\mid\mathrm{Re}\,w\leq-c\})
\]
and consider the relative homology group $H_{k}(X_{z},B_{c};\cL_{z})$.
Then for $c\leq c'$, natural inclusion map $i_{c,c'}:(X_{z},B_{c'})\hookrightarrow(X_{z},B_{c})$
induces a homomorphism $(i_{c,c'})_{*}:H_{k}(X_{z},B_{c'};\cL_{z})\to H_{k}(X_{z},B_{c};\cL_{z})$.
These give a projective system with a directed set $\R$ and we have
\[
H_{k}^{\Phi_{z}}(X_{z};\cL_{z})=\mathop{\lim_{\longleftarrow}}H_{k}(X_{z},B_{c};\cL_{z}).
\]
From the fact that $g_{z}:g_{z}^{-1}(\C\setminus D)\to\C\setminus D$
is a locally trivial fibration, we see that $H_{k}^{\Phi_{z}}(X_{z};\cL_{z})\simeq H_{k}(X_{z},B_{c};\cL_{z})$
for sufficiently large $c$. Then we take an $r(m-r)$-cycle of the
homology group $H_{r(m-r)}^{\Phi_{z}}(X_{z};\cL_{z})$ as $C(z)$
in (\ref{eq:confl-1}).

To obtain the Radon HGF, we must vary $C(z)$ so that it depends continuously
on $z$. Put $\mathscr{X}:=\cup_{z\in Z}X_{z}=\{([t],z)\mid\det(tz^{(k)})\neq0,\;1\leq k\leq\ell\}\subset T\times Z$,
which is a Zariski open subset of the smooth algebraic variety $T\times Z$.
Let $\pi:\mathscr{X}\to\C\times Z$ be a morphism defined by $\pi([t],z)=(g(t,z),z)$.
Applying again Proposition \ref{prop:Verdier} to this morphism, there
is an algebraic hypersurface $\Sigma\subset\C\times Z$ such that
the map $\pi$ gives a locally trivial fibration on the inverse image
$\pi^{-1}((\C\times Z)\setminus\Sigma)$ of the complement of $\Sigma$.
We take a Zariski open subset $U\subset Z$ consisting of points $z\in Z$
such that $\Sigma\cap(\C\times\{z\})$ is a finite set. Then we have
a local system
\[
\bigcup_{z\in U}H_{r(m-r)}^{\Phi_{z}}(X_{z};\cL_{z})\to U.
\]
We take its local section as $C=\{C(z)\}$ to obtain the Radon HGF
of type $\lm$. 

We need the following important property for the Radon HGF which states
the covariance of the function under the action of $\GL m\times\hlam$
on $Z$. The action
\[
\GL m\times\mnm\times\hlam\ni(g,z,h)\mapsto gzh\in\mnm
\]
induces that on the set $Z$.
\begin{prop}
\label{prop:covariance}For the Radon HGF of type $\lm$, we have
the formulae

(1) $F_{\lm}(gz,\al;C)=\det(g)^{-r}F_{\lm}(z,\al;\tilde{C}),\quad g\in\GL m,$

(2) $F_{\lm}(zh,\al;C)=F_{\lm}(z,\al;C)\chi_{\lm}(h;\al),\quad h\in\tilde{H}_{\lm}$.
\end{prop}

\begin{proof}
Write $F$ for $F_{\lm}$ for the sake of simplicity. We show (2).
\begin{align*}
F(zh,\al;C) & =\int_{C(zh)}\chi_{\lm}(t(zh);\al)\cdot\tau(t)=\int_{C(zh)}\chi_{\lm}((tz)h;\al)\cdot\tau(t)\\
 & =\int_{C(z)}\chi_{\lm}(tz;\al)\cdot\chi_{\lm}(h;\al)\cdot\tau(t)=\left(\int_{C(z)}\chi_{\lm}(tz;\al)\cdot\tau(t)\right)\chi_{\lm}(h;\al)\\
 & =F(z,\al;C)\chi_{\lm}(h;\al).
\end{align*}
Here at the third equality, we used the fact that $\chi_{\lm}$ is
a character. Next we show (1). In 
\[
F(gz,\al;C)=\int_{C(gz)}\chi_{\lm}(t(gz);\al)\cdot\tau(t)=\int_{C(gz)}\chi_{\lm}((tg)z;\al)\cdot\tau(t),
\]
we make a change of variable $t\mapsto s:=tg$. Then we show that
$\tau(t)=\tau(sg^{-1})=(\det g)^{-r}\tau(s)$. Note that the form
$\tau$ is determined uniquely up to multiplication of constant. Since
$t\mapsto s=tg$ defines an automorphism of the Grassmannian $T$,
$\tau(s)=\rho(g)\tau(t)$ for some character $\rho:\GL m\to\cbatu$,
which must have the form $\rho(g)=(\det g)^{p}$ for some $p\in\Z$.
We find the integer $p$ taking $g=aI_{m}$. Using the expression
$\tau(t)=(\det t')^{m}du$, as mentioned in (\ref{eq:char-2-1}),
we see that
\[
\tau(s)=\tau(tg)=(\det(t'a))^{m}du=a^{rm}\tau(t).
\]
On the other hand $\rho(g)=\rho(a1_{m})=\det(a1_{m})^{p}=a^{mp}$.
So we have $a^{rm}=a^{mp}$ for any $a\in\cbatu$, hence $p=r$. Thus
\begin{align*}
F(gz,\al;C) & =\int_{C(gz)}\chi_{\lm}((tg)z;\al)\cdot\tau(t)=\int_{\tilde{C}(z)}\chi_{\lm}(sz;\al)\cdot\tau(sg^{-1})\\
 & =(\det g)^{-r}\int_{\tilde{C}(z)}\chi_{\lm}(sz;\al)\cdot\tau(s)=(\det g)^{-r}F(z,\al;\tilde{C}),
\end{align*}
where $\tilde{C}(z)$ is the chain obtained from $C(gz)$ as the image
of the map $T\ni[t]\mapsto[tg]\in T$.
\end{proof}

\section{\label{sec:Radon-hypergeometric-system}Radon hypergeometric system}

In this section we mention about the system of differential equations
satisfied by the Radon HGF, which we call the Radon hypergeometric
system (Radon HGS). The system consists of the differential equations
characterizing the image of Radon transform, which form a main body
of the system, and of the equations which are infinitesimal form of
the covariance property of Radon HGF given in Proposition \ref{prop:covariance}.
The Radon HGS for the non-confluent case was already given by Oshima
\cite{Ohsima}. For the HGS on Hermitian symmetric spaces, see \cite{Tanisaki}.
The important part of their work is that the differential equations
\emph{characterizing} the image of Radon transform are given. An easy
part is that functions given by a Radon transform satisfy the differential
equations used for the characterization. We give here the Radon HGS
and an elementary proof of the easy part for the sake of completeness
of presentation.

For the matrix $z\in\mat(m,N)$ of independent variables, its entries
are indexed as $z=(z_{i,j})_{1\leq i\leq m,1\leq j\leq N}$. Let $D_{i,j}$
be the differentiation with respect to $z_{i,j}$: $D_{i,j}=\pa/\pa z_{i,j}$.
Put $[1,m]:=\{i\in\Z\mid1\leq i\leq m\}$ and $[1,N]:=\{j\in\Z\mid1\leq j\leq N\}$.
Take subsets $I\subset[1,m]$ and $J\subset[1,N]$ of cardinality
$|I|=|J|=r+1$ and let $I=\{i_{1}<\cdots<i_{r+1}\}$, $J=\{j_{1}<\cdots<j_{+1}\}$.
For such $I,J$, define the $(r+1)$-th order differential operator
\[
D_{I,J}=\det\left(D_{i_{\mu},j_{\nu}}\right)_{1\leq\mu,\nu\leq r+1}.
\]

\begin{prop}
The Radon HGF $F_{\lm}(z,\al;C)$ satisfies the differential equations
\begin{align*}
(i)\quad & D_{I,J}F=0\quad\text{for}\;\forall I\subset[1,m],\forall J\subset[1,N],|I|=|J|=r+1,\\
(ii)\quad & \text{infinitesimal form of }F(zh)=\chi_{\lm}(h;\al)F(z),\quad h\in\tilde{H_{\lm}},\\
(iii)\quad & \text{infinitesimal form of }F(gz)=(\det g)^{-r}F(z),\quad g\in\GL m.
\end{align*}
\end{prop}

\begin{proof}
Note that $\chi_{\lm}$ is regarded as a homogeneous multivalued function
of degree $-m$ on the space $\mat'(r,N)$ of homogeneous coordinates
of $M_{1}=\gras(r,N)$ via the map $\iota_{\lm}:H_{\lm}\to\mat'(r,N)$
given by (\ref{eq:radon-conf-3}). Let $v=(v_{i,j})$ be the homogeneous
coordinates of $M_{1}$. Put $\pa_{i,j}:=\pa/v_{i,j}$. For a given
homogeneous (multivalued) function $f(v)$ of degree $-m$ on $\mat'(r,N)$.
Define $\Phi(z)=\int_{C}f(tz)\cdot\tau(t)$. Note that $f$ is $\chi_{\lm}$
in our case. We show that $D_{I,J}\Phi=0$ for any subset $I\subset[1,m],J\subset[1,N],|I|=|J|=r+1$.
Note that 
\[
D_{I,J}\Phi(z)=\int_{C}D_{I,J}f(tz)\cdot\tau
\]
and $D_{I,J}=\sum_{\sigma\in\Si_{r+1}}(\sgn\,\sigma)D_{i_{\sigma(1)},j_{1}}\cdots D_{i_{\sigma(r+1)},j_{r+1}}$.
Note also that 
\[
D_{i,j}f(tz)=\frac{\pa}{\pa z_{i,j}}f(tz)=\sum_{k=1}^{r}\frac{\pa f}{\pa v_{k,j}}(tz)\cdot t_{k,i}=\left(\sum_{k=1}^{r}t_{k,i}\pa_{k,j}f\right)_{v=tz}
\]
It follows that 
\begin{align*}
D_{I,J}f(tz) & =\left[\sum_{\sigma\in\Si_{r+1}}(\sgn\,\sigma)\left(\sum_{k=1}^{r}t_{k,i_{\sigma(1)}}\pa_{k,j_{1}}\right)\cdots\left(\sum_{k=1}^{r}t_{k,i_{\sigma(r+1)}}\pa_{k,j_{r+1}}\right)f\right]_{u=tz}\\
 & =\left[\det\left(\sum_{k=1}^{r}t_{k,i_{p}}\pa_{k,j_{q}}\right)_{1\leq p,q\leq r+1}f\right]_{u=tz}.
\end{align*}
The differential operator in the last line is written as 
\[
\det\left[\left(\begin{array}{ccc}
t_{1,i_{1}} & \dots & t_{r,i_{1}}\\
\vdots &  & \vdots\\
t_{1,i_{r+1}} & \dots & t_{r,i_{r+1}}
\end{array}\right)\left(\begin{array}{ccc}
\pa_{1,j_{1}} & \dots & t_{1,j_{r+1}}\\
\vdots &  & \vdots\\
t_{r,j_{1}} & \dots & t_{r,j_{r+1}}
\end{array}\right)\right]
\]
and it turns out to be zero since it is the determinant of the product
of matrices of size $(r+1)\times r$ and of size $r\times(r+1)$ and
each matrix has the rank at most $r$. Hence we have shown that $D_{I,J}\Phi(z)$=0.
\end{proof}

\section{\label{sec:Radon-HGF-as}Radon HGF as an extension of classical HGF}

In this section, we make clear the relation of the Radon HGF to some
of the classical HGFs and of the HGFs defined by Hermitian matrix
integral. We begin by recalling these classical HGFs and their Hermitian
matrix integral analogues.

\subsection{\label{subsec:Classical-HGF-matrix}Classical HGF and its Hermitian
matrix integral analogue}

\subsubsection{Beta, Gamma and Gaussian integral}

One may say that among the member of the classical HGF family, the
simplest and most fundamental ones are the beta function, the gamma
function, and the Gaussian integral:
\begin{align*}
B(a,b) & :=\int_{0}^{1}u^{a-1}(1-u)^{b-1}du,\\
\G(a) & :=\int_{0}^{\infty}e^{-u}u^{a-1}du,\\
G & :=\int_{-\infty}^{\infty}e^{-\frac{1}{2}u^{2}}du=\sqrt{2\pi}.
\end{align*}
It is well known that $B(a,b)$ converges for $(a,b)\in\C^{2}$ such
that $\re\,a,\re\,b>0$ and $\G(a)$ for $\re\,a>0$. The Gaussian
integral is not a function. However we consider it is natural to include
it in one group as one may understand in the subsequent sections. 

The Hermitian integral analogue is known for these integrals. Let
$\herm$ be the set of $r\times r$ complex Hermitian matrices. It
is a real vector space of dimension $r^{2}$. For $U=(U_{i,j})\in\herm,$
let $dU$ be a volume form on $\herm$, which is the usual Euclidean
volume form given by

\begin{equation}
dU=\bigwedge_{i=1}^{r}dU_{i,i}\bigwedge_{i<j}\left(d\re(U_{i,j})\wedge d\im(U_{i,j})\right).\label{eq:hermInt-1}
\end{equation}
Note that this form can be written as 
\begin{equation}
dU=\left(\frac{\sqrt{-1}}{2}\right)^{r(r-1)/2}\bigwedge_{i=1}^{r}dU_{i,i}\bigwedge_{i\neq j}(dU_{i,j}\wedge dU_{j,i}).\label{eq:hermInt-2}
\end{equation}

The matrix integral version of the beta function, the gamma function
and the Gaussian integral are defined by 
\begin{align}
B_{r}(a,b) & :=\int_{0<U<1_{r}}(\det U)^{a-r}(\det(1_{r}-U))^{b-r}\,dU,\nonumber \\
\G_{r}(a) & :=\int_{U>0}\etr(-U)(\det U)^{a-r}\,dU,\label{eq:hermInt-3}\\
G_{r} & :=\int_{\herm}\etr(-\frac{1}{2}U^{2})\,dU,\nonumber 
\end{align}
respectively, where $\etr(U)=\exp(\Tr(U))$, $\Tr(U)$ being the trace
of $U$. The domain of integration is the set of positive definite
Hermitian matrices $U>0$ for the gamma, and the subset of $\herm$
satisfying $U>0$ and $1_{r}-U>0$ for the beta. It can be shown that
the gamma integral converges for $\re(a)>r-1$ and the beta integral
converges for $\re(a)>r-1,\re(b)>r-1$, and they define holomorphic
functions there. It is seen that $B_{r},\G_{r}$ and $G_{r}$ reduces
to the classical beta, gamma and Gaussian integral when $r=1$, respectively.
The following result for $B_{r}$ and $\G_{r}$ is known\cite{Faraut}.
\begin{prop*}
The following formulas hold.
\begin{align*}
(i)\quad & \G_{r}(a)=\pi^{\frac{r(r-1)}{2}}\prod_{i=1}^{r}\G(a-i+1).\\
(ii)\quad & B_{r}(a,b)=\frac{\G_{r}(a)\G_{r}(b)}{\G_{r}(a+b)}.
\end{align*}
\end{prop*}

\subsubsection{\label{subsec:class-HGF-matrix-2}Gauss HGF and its confluent family}

Next we consider the Gauss HGF and its confluent family: Kummer's
confluent HGF, Bessel function, Hermite-Weber function and Airy function.
They form an important class in the theory of special functions. The
Gauss HGF is defined by the power series
\[
\,_{2}F_{1}(a,b,c;x)=\sum_{m=0}^{\infty}\frac{(a)_{m}(b)_{m}}{(c)_{m}m!}x^{m}
\]
under the condition $c\notin\Z_{\leq0}$, represented by the complex
integral
\begin{equation}
\text{Gauss}:\;\frac{\G(c)}{\G(a)\G(c-a)}\int_{0}^{1}u^{a-1}(1-u)^{c-a-1}(1-ux)^{-b}du,\label{eq:hermInt-4}
\end{equation}
and satisfies the differential equation in the complex domain:
\[
x(1-x)y''+\{c-(a+b+1)x\}y'-aby=0.
\]
Any other solution of this equation can be obtained by the integral
(\ref{eq:hermInt-4}) taking an appropriate path of integration $C$
in $u$-plane instead of $\overrightarrow{0,1}$ . For the other members
of the confluent family of Gauss, the situation is similar. They are
characterized as a solution of the second order differential equations
and the solutions are given by the integral:
\begin{align*}
\text{Kummer}:\quad & \int_{C}e^{xu}u^{a-1}(1-u)^{c-a-1}du,\\
\text{Bessel}:\quad & \int_{C}e^{xu-\frac{1}{u}}u^{c-1}du,\\
\text{Hermite-Weber}:\quad & \int_{C}e^{xu-\frac{1}{2}u^{2}}u^{-c-1}dt,\\
\text{Airy}:\quad & \int_{C}e^{xu-\frac{1}{3}u^{3}}dt.
\end{align*}

A Hermitian integral analogue of these is considered in various contexts
\cite{Faraut,inamasu-ki,kimura-1,Kontsevich,Mehta,muirhead-2}. They
are functions of $X\in\herm$ given by

\begin{align}
\text{Gauss}:\quad & \frac{\G_{r}(c)}{\G_{r}(a)\G_{r}(c-a)}\int_{0<U<1_{r}}(\det U)^{a-r}(\det(1_{r}-U))^{c-a-r}(\det(1_{r}-UX))^{-b}\,dU,\nonumber \\
\text{Kummer}:\quad & \int_{C}\etr(UX)(\det U)^{a-r}(\det(1_{r}-U))^{c-a-r}\,dU,\nonumber \\
\text{Bessel}:\quad & \int_{C}\etr(UX-U^{-1})(\det U)^{c-r}\,dU,\label{eq:hermInt-5}\\
\text{Hermite-Weber}:\quad & \int_{C}\etr\left(UX-\frac{1}{2}U^{2}\right)(\det U)^{-c-r}\,dU,\nonumber \\
\text{Airy}:\quad & \int_{C}\etr\left(UX-\frac{1}{3}U^{3}\right)\,dU.\nonumber 
\end{align}
Here we do not enter into the discussion of the domain of integration
$C$. It is known \cite{Faraut,kimura-1,muirhead} that each of the
matrix integrals gives a function of eigenvalues $x_{1},\dots,x_{r}$
of $X$ and satisfies the system of partial differential equations
of the second order. The system is holonomic and the dimension of
the solution space at generic points is $2^{r}$.

\subsubsection{Classical HGF of several variables}

As an example of classical HGF of several variables, we take Lauricella's
$F_{D}$:
\begin{align*}
F_{D}(a,b_{1},\dots,b_{p},c;x_{1},\dots,x_{p}) & =\sum_{m_{1},\dots,m_{p}=0}^{\infty}\frac{(a)_{m_{1}+\cdots+m_{p}}(b_{1})_{m_{1}}\cdots(b_{p})_{m_{p}}}{(c)_{m_{1}+\cdots+m_{p}}m_{1}!\cdots m_{p}!}x_{1}^{m_{1}}\cdots x_{p}^{m_{p}}\\
 & =\frac{\G(c)}{\G(a)\G(c-a)}\int_{0}^{1}u^{a-1}(1-u)^{c-a-1}\prod_{j=1}^{p}(1-ux_{j})^{-b_{j}}du.
\end{align*}
When $p=1$, it reduces to the Gauss HGF, and so it is one of the
extensions of the Gauss HGF to several independent variables case.
In Section \ref{subsec:Herm-Radon-3}, we show that an Hermitian integral
analogue of $F_{D}$ appears from the Radon HGF. For $F_{D}$ and
for the other classical HGF of several variables, see \cite{Appell-2,Erdelyi}.

\subsection{\label{subsec:Herm-Radon}HGF by Hermitian matrix integral and Radon
HGF}

We show that the classical HGFs and Hermitian matrix integral analogues
explained in Section \ref{subsec:Classical-HGF-matrix} are understood
as particular cases of Radon HGF. To establish a connection, let us
consider the Radon HGF in a more restricted situation. Namely we consider
the Radon HGF of type $\lm$ in the case $m=2r,N=nr$ and assume some
additional condition on the space of independent variable $z\in\mat'(2r,nr)$.
Note that $n\geq3$ since $N>m$ by assumption. 

Let a partition $\lm=(n_{1},\dots,n_{\ell})$ of $n$ be given. We
say that $\mu=(m_{1},\dots,m_{\ell})\in\Z_{\geq0}^{\ell}$ is a subdiagram
of $\lm$ of weight $2$ if it satisfies 

\[
0\leq m_{k}\leq n_{k}\quad(\forall k)\quad\mbox{and}\quad|\mu|:=m_{1}+\cdots+m_{\ell}=2.
\]
More explicitly, $\mu$ has the form either 
\begin{equation}
\mu=(0,\dots,0,\overset{i}{1},0,\dots,0,\overset{j}{1},0,\dots,0)\;\text{or }\mu=(0,\dots,0,\overset{i}{2},0,\dots,0).\label{eq:red-1}
\end{equation}
The first case means that $m_{i}=m_{j}=1$ and $m_{k}=0$ for $k\neq i,j$.
Using this notation we define a Zariski open subset $Z_{\lm}\subset\mat'(2r,nr)$
of the space of independent variables as follows. We write $z\in\mat'(2r,nr)$
as $z=(z^{(1)},\dots,z^{(\ell)})$ arraying blocks, where the $j$-th
block $z^{(j)}$ is a $2r\times n_{j}r$ matrix. For any $j$, $z^{(j)}$
is also written as a block matrix:
\[
z^{(j)}=(z_{0}^{(j)},\dots,z_{n_{j}-1}^{(j)})=\left(\begin{array}{ccc}
z_{0,0}^{(j)} & ,\dots, & z_{0,n_{j}-1}^{(j)}\\
z_{1,0}^{(j)} & ,\dots, & z_{1,n_{j}-1}^{(j)}
\end{array}\right),\quad z_{p,q}^{(j)}\in\mat(r).
\]
For a subdiagram $\mu\subset\lm$ with $|\mu|=2$, $\mu$ is either
of the form in (\ref{eq:red-1}). According as the form of $\mu$,
put
\[
z_{\mu}=(z_{0}^{(i)},z_{0}^{(j)})\;\text{or }z_{\mu}=(z_{0}^{(i)},z_{1}^{(i)}).
\]
Then $Z_{\lm}$ is defined as 
\[
Z_{\lm}:=\{z\in\mat(2r,nr)\mid\det z_{\mu}\neq0\;\text{for any subdiagram}\;\mu\subset\lm,|\mu|=2\}.
\]
It is easily seen that $Z_{\lm}$ is invariant by the action $\GL{2r}\curvearrowright\mat'(2r,nr)\curvearrowleft H_{\lm}$.
Taking into account the covariance property for the Radon HGF with
respect to the action of $\GL{2r}\times H_{\lm}$ given in Proposition
\ref{prop:covariance}, we try to take the independent variable $z$
to a simpler form $\bx\in Z_{\lm}$ which gives a representative of
the orbit $O(z)$ of $z$.

\subsubsection{\label{subsec:herm-radon-1}Radon HGF for $(m,N)=(2r,3r)$}
\begin{lem}
\label{lem:Her-Ra-1}Let $\lm$ be a partition of $3$. For any $z\in Z_{\lm}$,
we can take a representative $\bx\in Z_{\lm}$ of the orbit $O(z)$
as given in the following table. 

\bigskip

\begin{tabular}{|c|c|}
\hline 
$\lm$ & normal form $\bx$\tabularnewline
\hline 
\hline 
$(1,1,1)$ & $\left(\begin{array}{ccc}
1_{r} & 0 & 1_{r}\\
0 & 1_{r} & -1_{r}
\end{array}\right)$\tabularnewline
\hline 
$(2,1)$ & $\left(\begin{array}{ccc}
1_{r} & 0 & 0\\
0 & 1_{r} & 1_{r}
\end{array}\right)$\tabularnewline
\hline 
$(3)$ & $\left(\begin{array}{ccc}
1_{r} & 0 & 0\\
0 & 1_{r} & 0
\end{array}\right)$\tabularnewline
\hline 
\end{tabular}
\end{lem}

\begin{proof}
We prove the case $\lm=(1,1,1)$. Write $z\in Z_{\lm}$ as 
\[
z=(z_{1},z_{2},z_{3})=\left(\begin{array}{ccc}
z_{0,1} & z_{0,2} & z_{0,3}\\
z_{1,1} & z_{1,2} & z_{1,3}
\end{array}\right),\quad z_{i,j}\in\mat(r).
\]
Since $\det(z_{1},z_{2})\neq0$ for $z\in Z_{\lm}$. Put $g_{1}=(z_{1},z_{2})^{-1}\in\GL{2r}$,
then $g_{1}z$ has the form
\[
g_{1}z=\left(\begin{array}{ccc}
1_{r} & 0 & v_{0}\\
0 & 1_{r} & v_{1}
\end{array}\right),\quad v_{0},v_{1}\in\mat(r).
\]
Since $Z_{\lm}$ is invariant by the action of $\GL{2r}\times H_{\lm}$,
we see that $g_{1}z\in Z_{\lm}$. This condition is $v_{0},v_{1}\in\GL r$.
Take $h=\diag(1_{r},h_{2},h_{3})\in H_{\lm}$ and $g_{2}=\diag(1_{r},h_{2}^{-1})\in\GL{2r}$.
Then 
\[
g_{2}g_{1}zh=\left(\begin{array}{ccc}
1_{r} & 0 & v_{0}h_{3}\\
0 & 1_{r} & h_{2}^{-1}v_{1}h_{3}
\end{array}\right).
\]
So we can determine $h_{2},h_{3}\in\GL r$ by the condition $v_{0}h_{3}=1_{r},h_{2}^{-1}v_{1}h_{3}=-1_{r}$.
Hence, $z$ is taken to the normal form $\bx$ given in the table
of the lemma by the action of $\GL{2r}\times H_{\lm}$. 

Next we consider the case $\lm=(2,1)$. Write $z\in Z_{\lm}$ as $z=(z_{0}^{(1)},z_{1}^{(1)},z_{0}^{(2)})$
with $z_{i}^{(k)}\in\mat(2r,r)$. The condition for $z\in Z_{\lm}$
is $\det(z_{0}^{(1)},z_{1}^{(1)})\neq0$, $\det(z_{0}^{(1)},z_{0}^{(2)})\neq0$.
Put $g_{1}=(z_{0}^{(1)},z_{0}^{(2)})^{-1}\in\GL{2r}$ and consider
$g_{1}z$. Then 

\[
g_{1}z=\left(\begin{array}{ccc}
1_{r} & v_{0} & 0\\
0 & v_{1} & 1_{r}
\end{array}\right),\quad v_{0},v_{1}\in\mat(r).
\]
 The condition for $g_{1}z$ being contained in $Z_{\lm}$ is $v_{1}\in\GL r$.
Take $g_{2}\in\GL{2r}$ and $h\in H_{\lm}$ as 
\[
g_{2}=\left(\begin{array}{cc}
1_{r}\\
 & h_{2}^{-1}
\end{array}\right),\;h=\left(\begin{array}{ccc}
1_{r} & h_{1}\\
 & 1_{r}\\
 &  & h_{2}
\end{array}\right),\quad h_{2}\in\GL r.
\]
Then we have 
\[
g_{2}g_{1}zh=\left(\begin{array}{ccc}
1_{r} & v_{0}+h_{1}v_{1} & 0\\
0 & h_{2}^{-1}v_{1} & 1_{r}
\end{array}\right).
\]
So we determine $h$ by $v_{0}+h_{1}v_{1}=0$ and $h_{2}^{-1}v_{1}=1_{r}$,
which is possible since $v_{1}\in\GL r$. Thus, by the action of $\GL{2r}\times H_{(2,1)}$,
any $z\in Z_{(2,1)}$ can be taken to the normal form $\bx$ as given
in the lemma. The last case $\lm=(3)$ is similarly shown.
\end{proof}
For the normal form $\bx\in Z_{\lm}$ given in Lemma \ref{lem:Her-Ra-1},
we write down the integral for the corresponding Radon HGF. The space
of integration variable is $T=\gras(r,2r)\simeq\GL r\backslash\mat'(r,2r)$
with the homogeneous coordinates $t=(t',t'')\in\mat'(r,2r)$, $t',t''\in\mat(r)$.
Let $U=\{[t]\in\gras(r,2r)\mid\det t'\neq0\}$ be the affine neighbourhood
and let $u=(u_{i,j})\in\mat(r)$ be the affine coordinates defined
by $t=t'(1_{r},u)$, $u=(t')^{-1}t''$. Then we know that 
\[
F_{\lm}(z,\al;C)=\int_{C}\chi_{\lm}(\vec{u}z;\al)du,\quad\vec{u}=(1_{r},u),\;du=\wedge du_{i,j}.
\]
Then we have 

\begin{align}
F_{(1,1,1)}(\bx,\al;C) & =\int_{C}(\det u)^{\al_{2}}(\det(1-u))^{\al_{3}}du,\nonumber \\
F_{(2,1)}(\bx,\al;C) & =\int_{C}e^{\al_{2}\Tr(u)}(\det u)^{\al_{3}}du,\label{eq:her-ra-4}\\
F_{(3)}(\bx,\al;C) & =\int_{C}e^{\al_{2}\Tr(u)+\al_{3}\Tr(-\frac{1}{2}u^{2})}du.\nonumber 
\end{align}
For example we check the expression for $F_{(2,1)}(\bx,\al;C)$. Note
that the character is 
\begin{align*}
\chi_{(2,1)}(h;\al) & =(\det h_{0}^{(1)})^{\al_{1}}\exp\left(\al_{2}\Tr\left((h_{0}^{(1)})^{-1}h_{1}^{(1)}\right)\right)\cdot(\det h_{0}^{(2)})^{\al_{3}}
\end{align*}
and 
\[
\vec{u}\bx=(1_{r},u)\left(\begin{array}{ccc}
1_{r} & 0 & 0\\
0 & 1_{r} & 1_{r}
\end{array}\right)=(1_{r},u,u).
\]
Hence $\chi_{(2,1)}(\vec{u}\bx;\al)=(\det1_{r})^{\al_{1}}e^{\al_{2}\Tr(u)}(\det u)^{\al_{3}}=e^{\al_{2}\Tr(u)}(\det u)^{\al_{3}}$,
which is the integrand of $F_{(2,1)}(\bx,\al;C)$. Here we choose
the parameter $\al$ as $\al_{2}=-1$ for $F_{(2,1)}$. For $F_{(3)}$,
we choose $\al$ as $\al_{2}=0,\al_{3}=1$. Then the integrand has
the same form as the Hermitian matrix integral (\ref{eq:hermInt-3})
for $B_{r}(a,b),\G_{r}(a)$ and $G_{r}$. As for the space $\herm$
of integration variables for the Hermitian matrix integral, we remark
that $\herm$ is a real form of $\mat(r)$ in the sense that any $u\in\mat(r)$
can be written as $u=X+\sqrt{-1}Y$, $X,Y\in\herm$ uniquely. Moreover
we see from (\ref{eq:hermInt-2}) that the form $du$ coincides with
$dU$ modulo constant when $u$ is restricted to $\herm$. Hence in
the integrals (\ref{eq:her-ra-4}), we can choose $r^{2}$-chain in
the space $\herm$. So, in the case $|\lm|=3$, $F_{\lm}(\bx,\al;C)$
gives the beta function, the gamma function, and the Gaussian integral
defined in terms of the Hermitian matrix integral corresponding to
the partitions $(1,1,1),(2,1)$ and $(3)$, respectively.
\begin{rem}
\label{rem:Her-Ra}In the above, we assumed additional conditions
on $\al$ in the cases $\lm=(2,1)$ and $\lm=(3)$, which seems a
little bit artificial. As will be shown in the forthcoming paper,
we can carry out this kind of normalization of $\al$ as an effect
of the action of Weyl group analogue, which describes the symmetry
of the Radon HGF.
\end{rem}

\subsubsection{\label{subsec:herm-radon-2}Radon HGF for $(m,N)=(2r,4r)$}
\begin{lem}
\label{lem:Her-Ra-2}Let $\lm$ be a partition of $4$. For any $z\in Z_{\lm}$,
we can take a representative $\bx\in Z_{\lm}$ of the orbit $O(z)$
of the action of $\GL{2r}\times H_{\lm}$ as given in the following
table.

\bigskip

\begin{tabular}{|c|c|c|c|}
\hline 
$\lm$ & $\bx_{1}$ & $\bx_{2}$ & $\bx_{3}$\tabularnewline
\hline 
\hline 
$(1,1,1,1)$ & $\left(\begin{array}{cccc}
1_{r} & 0 & 1_{r} & 1_{r}\\
0 & 1_{r} & -1_{r} & -x
\end{array}\right)$ &  & \tabularnewline
\hline 
$(2,1,1)$ & $\left(\begin{array}{cccc}
1_{r} & 0 & 0 & 1_{r}\\
0 & x & 1_{r} & -1_{r}
\end{array}\right)$ & $\left(\begin{array}{cccc}
1_{r} & 0 & 0 & 1_{r}\\
0 & 1_{r} & 1_{r} & -x'
\end{array}\right)$ & $\ensuremath{\left(\begin{array}{cccc}
1_{r} & 0 & 0 & x''\\
0 & 1_{r} & 1_{r} & -1_{r}
\end{array}\right)}$\tabularnewline
\hline 
$(2,2)$ & $\left(\begin{array}{cccc}
1_{r} & 0 & 0 & 1_{r}\\
0 & x & 1_{r} & 0
\end{array}\right)$ & $\ensuremath{\left(\begin{array}{cccc}
1_{r} & 0 & 0 & x'\\
0 & 1_{r} & 1_{r} & 0
\end{array}\right)}$ & \tabularnewline
\hline 
$(3,1)$ & $\left(\begin{array}{cccc}
1_{r} & 0 & 0 & 0\\
0 & 1_{r} & x & 1_{r}
\end{array}\right)$ & $\left(\begin{array}{cccc}
1_{r} & x' & 0 & 0\\
0 & -1_{r} & 0 & 1_{r}
\end{array}\right)$ & $\left(\begin{array}{cccc}
1_{r} & 0 & 0 & x''\\
0 & 1_{r} & 0 & -1_{r}
\end{array}\right)$\tabularnewline
\hline 
$(4)$ & $\left(\begin{array}{cccc}
1_{r} & 0 & 0 & 0\\
0 & 1_{r} & 0 & x
\end{array}\right)$ &  & \tabularnewline
\hline 
\end{tabular}
\end{lem}

\begin{proof}
We check the case $\lm=(1,1,1,1)$ only. By the proof of Lemma \ref{lem:Her-Ra-1},
there exist $g\in\GL{2r}$ and $h=\diag(h_{1},h_{2},h_{3},1_{r})\in H_{\lm}$
such that 
\[
gzh=\left(\begin{array}{cccc}
1_{r} & 0 & 1_{r} & v_{0}\\
0 & 1_{r} & -1_{r} & v_{1}
\end{array}\right),\quad v_{0},v_{1}\in\mat(r).
\]
The condition for $gzh$ being an element of $Z_{\lm}$ is $\det v_{0}\neq0,\det v_{1}\neq0,\det(v_{0}-v_{1})\neq0$.
Put $h'=\diag(1_{r},1_{r},1_{r},v_{0}^{-1})$, then
\[
gzhh'=\left(\begin{array}{cccc}
1_{r} & 0 & 1_{r} & 1_{r}\\
0 & 1_{r} & -1_{r} & v_{1}v_{0}^{-1}
\end{array}\right).
\]
Then, putting $x:=-v_{1}v_{0}^{-1}$, we obtain the normal form $\bx_{1}$
given in the table. 
\end{proof}
Corresponding to the normal form $\bx_{i}\in Z_{\lm}$ given above,
we have the following Radon HGF in terms of the integration variable
$u=(u_{i,j})_{1\leq i,j\leq r}$ and $du=\wedge_{i<j}du_{i,j}$:

(1) $\lm=(1,1,1,1)$:
\[
F_{(1,1,1,1)}(\bx_{1},\al;C)=\int_{C}(\det u)^{\al_{2}}(\det(1-u))^{\al_{3}}(\det(1-ux))^{\al_{4}}du.
\]

(2) $\lm=(2,1,1)$:
\begin{align*}
\text{\ensuremath{F_{(2,1,1)}}(\ensuremath{\bx_{1}},\ensuremath{\al};C) } & =\int_{C}e^{\al_{2}\Tr(ux)}(\det u)^{\al_{3}}(\det(1-u))^{\al_{4}}du,\\
\text{\ensuremath{F_{(2,1,1)}}(\ensuremath{\bx_{2}},\ensuremath{\al};C) } & =\int_{C}e^{\al_{2}\Tr(u)}(\det u)^{\al_{3}}(\det(1-ux'))^{\al_{4}}du,\\
\text{\ensuremath{F_{(2,1,1)}}(\ensuremath{\bx_{3}},\ensuremath{\al};C) } & =\int_{C}e^{\al_{2}\Tr(u)}(\det u)^{\al_{3}}(\det(x''-u))^{\al_{4}}du.
\end{align*}

(3) $\lm=(2,2)$:
\begin{align*}
\text{\ensuremath{F_{(2,2)}}(\ensuremath{\bx_{1}},\ensuremath{\al};C) } & =\int_{C}e^{\al_{2}\Tr(ux')}(\det u)^{\al_{3}}e^{\al_{4}\Tr(u^{-1})}du,\\
\text{\ensuremath{F_{(2,2)}}(\ensuremath{\bx_{2}},\ensuremath{\al};C)} & =\int_{C}e^{\al_{2}\Tr(u)}(\det u)^{\al_{3}}e^{\al_{4}\Tr(u^{-1}x')}du.
\end{align*}

(4) $\lm=(3,1)$:
\begin{align*}
\text{\ensuremath{F_{(3,1)}}(\ensuremath{\bx_{1}},\ensuremath{\al};C) } & =\int_{C}e^{\al_{2}\Tr(u)+\al_{3}\Tr(ux-\frac{1}{2}u^{2})}(\det u)^{\al_{4}}du,\\
\text{\ensuremath{F_{(3,1)}}(\ensuremath{\bx_{2}},\ensuremath{\al};C) } & =\int_{C}e^{\al_{2}\Tr(x'-u)+\al_{3}\Tr(-\frac{1}{2}(x'-u)^{2})}(\det u)^{\al_{4}}du,\\
\text{\ensuremath{F_{(3,1)}}(\ensuremath{\bx_{3}},\ensuremath{\al};C) } & =\int_{C}e^{\al_{2}\Tr(u)+\al_{3}\Tr(-\frac{1}{2}u^{2})}(\det(x''-u))^{\al_{4}}du.
\end{align*}

(5) $\lm=(4)$ case:

\[
F_{(4)}(\bx_{1},\al;C)=\int e^{\al_{1}\Tr(u)+\al_{2}\Tr(-\frac{1}{2}u^{2})+\al_{3}\Tr(ux-\frac{1}{3}u^{3})}du.
\]
In Lemma \ref{lem:Her-Ra-2}, there are partitions $\lm$ for which
we listed up multiple normal forms $\bx_{i}$. Corresponding to normal
forms, we obtain different expressions of the Radon HGF in appearance.
However they are related to each other through the covariance property
given in Proposition \ref{prop:covariance}.

To relate the Radon HGF $F_{\lm}(\bx_{1},\al;C)$ to the HGFs given
by Hermitian matrix integral (\ref{eq:hermInt-5}), which are analogues
of the Gauss HGF and its confluent family, we add the following condition
on the parameter $\al=(\al_{1},\al_{2},\al_{3},\al_{4})\in\C^{4}$:
\begin{align*}
(2,1,1): & \quad\al_{2}=1,\\
(2,2): & \quad\al_{2}=1,\al_{4}=-1,\\
(3,1): & \quad\al_{2}=0,\al_{3}=1,\\
(4): & \quad\al_{2}=\al_{3}=0,\al_{4}=1.
\end{align*}
Under the above choice of $\al$, the Radon HGF $F_{\lm}(\bx_{1},\al;C)$
of type $\lm$ gives the Hermitian matrix integral analogue of Gauss,
Kummer, Bessel, Hermite-Weber and Airy as the case $\lm=(1,1,1,1)$,
$(2,1,1),(2,2)$, $(3,1)$ and $(4)$, respectively. Remark \ref{rem:Her-Ra}
explains the reason of the above choice of parameter $\al$ in this
case too.

\subsubsection{\label{subsec:Herm-Radon-3}Radon HGF for $(m,N)=(2r,nr)$ with $\protect\lm=(1,\dots,1)$}
\begin{lem}
\label{lem:normal-nonconf}For any $z\in Z_{(1,\dots,1)}$, we can
take a representative $\bx\in Z_{(1,\dots,1)}$ of the orbit $O(z)$
as 
\[
\bx=\left(\begin{array}{cccccc}
1_{r} & 0 & 1_{r} & 1_{r} & \dots & 1_{r}\\
0 & 1_{r} & -1_{r} & -x_{4} & \dots & -x_{n}
\end{array}\right),\quad x_{j}\in\mat(r).
\]
The condition for $\bx$ to belong to $Z_{(1,\dots,1)}$ is 
\[
\det x_{i}\neq0,\;\det(1_{r}-x_{i})\neq0,\;\det(x_{i}-x_{j})\neq0\quad4\leq\forall i\ne j\leq n.
\]
\end{lem}

The proof of this lemma is similar to that of Lemma \ref{lem:Her-Ra-2}.
Under Assumption \ref{assu:radon-nonconf}, we have 
\begin{equation}
F(\bx,\al;C)=\int_{C(\bx)}(\det u)^{\al_{2}}(\det(1_{r}-u))^{\al_{3}}\prod_{j=4}^{n}(\det(1_{r}-ux_{j})^{\al_{j}}du,\label{eq:exa-nonconf-1}
\end{equation}
 which gives a Hermitian matrix integral analogue of Lauricella's
$F_{D}$ of $n-3$ variables.

\end{document}